\documentclass[reqno,11pt]{amsart}
\usepackage{amssymb, amsmath,amsmath,latexsym,amssymb,amsfonts,amsbsy, amsthm,mathtools,graphicx,CJKutf8,CJKnumb,CJKulem,color,cases,empheq,extpfeil,bm}

\renewcommand{\theequation}{\thesection.\arabic{equation}}
\numberwithin{equation}{section}

\allowdisplaybreaks

\newcommand{\cal }{\mathcal }

\newcommand{\dif}{\,\mathrm{d}}
\newcommand\supp{\mathrm{supp}}

\newtheorem{remark}{Remark}[section]
\newtheorem{lemma}{Lemma}[section]

\newtheorem{cor}{Corollary}[section]
\newtheorem{prop}{Proposition}[section]
\newtheorem{theorem}{Theorem}[section]

\begin{document}


\title[Global well-posedness of second-grade fluid equations in 2D exterior domain]
{Global well-posedness of second-grade fluid equations in 2D exterior domain}

\author[You]{Xiaoguang You}
\address{Xiaoguang You $\newline$ School of Mathematics, Northwest University, Xi'an 710069, P. R. China}
\email{wiliam$\_$you@aliyun.com}

\author[Zang]{Aibin Zang}
\address{Aibin Zang $\newline$ School of Mathematics and Computer Science \& The Center of Applied Mathematics, Yichun university, Yichun, Jiangxi 340000, P. R. China}
\email{zangab05@126.com}


\begin{abstract}
In this article, we consider the second-grade fluid equations in 2D exterior domain $\Omega$ with Dirichlet boundary conditions. For initial data $\boldsymbol{u}_0 \in \bm{H}^3(\Omega)$, the second-grade fluid equations is shown to be globally well-posed. Furthermore, for arbitrary $T > 0$ and $s \geq 3$, we prove that the solution belongs to $L^\infty([0, T]; \bm{H}^s(\Omega))$ provided that $\bm{u}_0$ is in $\bm{H}^s(\Omega)$.
\end{abstract}


\maketitle


\vskip 0.2cm


\renewcommand{\theequation}{\thesection.\arabic{equation}}
\setcounter{equation}{0}

\section{Introduction}
There are many models of non-Newtonian fluids which have recently attracted wide attention. The second-grade fluid model is a well-known subclass of them and is admissible for slow flow fluids such as industrial fluids, slurries and polymer melts, we refer \cite{dunn1974thermodynamics,fosdick1979anomalous} for a comprehensive theory of second-grade fluids.  In the classical incompressible fluids of second  grade, the Cauchy stress tensor $\sigma$ is given by
\begin{align}
\nonumber \sigma = -pI + 2\nu A_1 + \alpha_1 A_2 + \alpha_2 A_1^2,
\end{align}
where $\nu$ is the kinematic viscosity, p is the indeterminate part of the stress due to the constraint of incompressibility,  $\alpha_1$ and $\alpha_2$ are material moduli which are usually referred to as the normal stress coefficients. $A_1$, $A_2$ stand for the first two Rivlin-Ericksen tensors, which are
\begin{align}
\nonumber&A_1 = \frac{1}{2}(\nabla \bm{u} + \nabla \bm{u}^t),\\
\nonumber&A_2 = \frac{DA_1}{Dt} + (\nabla\bm{u})^t A_1 + A_1 (\nabla \bm{u}),
\end{align}
where
\begin{align}
\nonumber\frac{D}{Dt} = \partial_t + \bm{u} \cdot \nabla
\end{align}
is the material derivative. In \cite{dunn1974thermodynamics}, Dunn and Fosdick showed that the normal stress coefficients $\alpha_1,\alpha_2$ must verify the inequality
\begin{align}
\nonumber\alpha_1 + \alpha_2 = 0; \alpha_1 \geq 0
\end{align}
as restrictions imposed by thermodynamics and the assumption that the specific Helmholtz free energy of the fluid be minimum in equilibrium.

Set $\alpha := \sqrt{\alpha_1}$. The substitution of the above Cauchy stress tensor $\sigma$ into the equation of linear momentum yields the following equations
\begin{align}
    \nonumber \partial_t (\bm{\bm{u} - \alpha^2\Delta\bm{u}}) + \bm{u} \cdot {\nabla} (\bm{u} - \alpha^2 \Delta \bm{u}) + (\nabla \bm{u})^t \cdot (\bm{u} - \alpha^2\Delta\bm{u}) + \nabla p = \nu \Delta \bm{u}.
\end{align}
We can see that the above second-grade fluid equations are an interpolant between the Navier–Stokes equations and the Euler-$\alpha$ equations. In what follows, we only consider the case $\alpha > 0$.

The well-posedness problem of second-grade fluid equations has been well studied by various authors. The existence and uniqueness of solutions to second-grade fluid equations in bounded domain with Dirichlet boundary conditions, was proved by Cioranescu and Ouazar \cite{cioranescu1984existence}. In the two-dimensional case the solution is global in time, and local for small data in the three-dimensional case. Later, Cioranescu and Girault \cite{cioranescu1997weak} showed that the solution in the three-dimensional case is actually global for small data.  In \cite{busuioc1999second}, Busuioc studied the existence of solutions for whole space $\mathbb{R}^n$(n=2, 3). She proved that there exists a local strong solution provided the initial data is sufficient smooth, and the solution is global for two dimension. On the other hand, Bresch and Lemoine \cite{bresch1997existence} established the existence and uniqueness of $W^{2,r} (r > 3)$ stationary solution for three dimensional bounded domain of class $C^2$ with smallness restrictions on the kinematic viscosity $\nu$. Recently, the existence and uniqueness of strong solution in the torus $\mathbb{T}^2$ was studied in \cite{paicu2013dynamics,paicu2012regularity}. For further results concerning the second-grade fluid equations, we refer the readers to \cite{bresch2020existence,galdi1997slow,fosdick1980thermodynamics,galdi1994further,galdi1993existence,girault2007time,girault1999analysis,oliver2001vortex,shkoller2001smooth}.

Let $\mathcal{O} \subset \mathbb{R}^2$ be a bounded, simply connected domain with $C^\infty$ Jordan boundary $\Gamma$. Without loss of generality, we set $\alpha=1$, and consider the second-grade fluid equations in exterior domain $\Omega=\mathbb{R}^2 \setminus \overline{\mathcal{O}}$,
given by:
\begin{numcases}{}
    \partial_t \bm{v} + \bm{u} \cdot {\nabla} \bm{v} + (\nabla \bm{u})^t \cdot \bm{v} + \nabla p = \nu \Delta \bm{u}
      &  $\text{in}  \ \Omega \times (0, \infty), $ \label{second-grade-rewriten-1}   \\
    \text{div} \ \bm{u} = 0 &  $\text{in} \ \Omega \times [0, \infty) $ \label{second-grade-rewriten-2},\\
    \bm{u} = 0  & $\text{on} \ \Gamma \times [0, \infty] \label{second-grade-rewriten-3}$,\\
    \bm{u}|_{t=0} = \bm{u}_0 & $\text{in} \ \Omega $\label{second-grade-rewriten-4},\\
    \bm{u}(x, t) \to 0 & \text{as} $|x| \to \infty, t \in [0, \infty), $ \label{second-grade-rewriten-5}
\end{numcases}
where $\bm{v} = \bm{u} - \Delta \bm{u}$, and $\bm{u}_0$ is the initial velocity. Above $\bm{u}$ is called the \textit{filtered} velocity, while $\bm{v}$ is the \textit{unfiltered} velocity.

In this paper, we will study the global existence and uniqueness to the above equations. Solving this problem is not easy, the major difficult arises from the lack of 'good' a priori estimates on the derivatives of the solution, which is mainly due, on one hand, to the rather involved form of the nonlinearities and high order derivatives in ($\ref{second-grade-rewriten-1}$), and on the other hand, to unboundedness of the domain.
Our work here is inspired by the techniques in \cite{2109.00915} that used to solve the well-posedness problem of Euler-$\alpha$ equations. By taking curl of equation $(\ref{second-grade-rewriten-1})$, we consider instead the vorticity-stream formula
\begin{numcases}{}
     \partial_t q + \bm{u} \cdot \nabla{q}  + \nu q - \nu \nabla^{\perp}\cdot \bm{u}= 0 & $ \text{in } \,\Omega \times [0,\infty]$, \label{vstream-1}\\
     \Delta_x {\psi}(x, t) = {q}(x, t) & $ \text{in } \,\Omega \times [0, \infty]$, \label{vstream-2} \\
     {\bm{u}}(x, t) - \Delta{\bm{u}}(x, t) + \nabla p = \nabla^\perp \psi &  $\text{in } \Omega \times [0, \infty]$,  \label{vstream-3}
\end{numcases}
here q(i.e. $\nabla^\perp \cdot \bm{v}$) is called the \textit{unfiltered} vorticity, while $\psi$ is the stream function. We observe that each equation of \eqref{vstream-1}--\eqref{vstream-3} is linear and classical. In fact, equation \eqref{vstream-1} is the transport equation, and for fixed $t\in [0, \infty)$, equation \eqref{vstream-2} is the Poisson equation, while \eqref{vstream-3} is the stationary Stokes equation.  In view of the well known results of these equations, we could obtain a priori estimate of $\bm{u}$ to establish local well-posedness. However, to show the solution is global, it is indispensable to utilize the estimates that exploited from the origin equations \eqref{second-grade-rewriten-1}--\eqref{second-grade-rewriten-5}. This is mainly due to that, in the Poisson equation \eqref{vstream-2}, low order derivatives of $\psi$ could not be bounded by its high order ones in exterior domain.

All in all, the main feature of the proof is based on constructing a family of approximate equations and a limit process using compactness arguments in order to control the nonlinear terms. We want to remark that the key point in our proof lies on the property that \textit{unfiltered} vorticity $q$ of the approximate equations is compactly supported in arbitrary fixed time interval provided that the initial data is compactly supported.

The remainder of this paper is organized as follows. In section 2, we will introduce notations and construct the approximate equations, then we will present our main results. In section 3 we will give some preliminaries and one technical lemma. In section 4, we will prove the global existence and uniqueness of solution to the approximate equations. In section 5, we will prove global well-posedness of original equations \eqref{second-grade-rewriten-1}--\eqref{second-grade-rewriten-5}. In section 6, some discussions and comments are given.

\renewcommand{\theequation}{\thesection.\arabic{equation}}
\setcounter{equation}{0}

\section{Notations and results} \label{notation}
Throughout the paper, if we denote by $C$ a positive constant with neither any subscript nor superscript then $C$ is considered as a generic constant whose value can change from line to line in the inequalities and depends on the parameters in question. On the other hand, we will denote in a bold character vector valued functions and in the usual scalar functions.

We will use standard notation for the Lebesgue spaces $L^p(\Omega)$ with the norm $\|\cdot\|_{L^p(\Omega)}$.  We use the notation $H^s(\Omega)$ for the usual $L^2$-based Sobolev spaces of order $s$. $C_0^\infty(\Omega)$ represents the space of smooth functions with infinitely many derivatives, compactly supported in $\Omega$, and $H_0^s(\Omega)$ the closure of $C_0^\infty$ under the $H^s$-norm. For the sake of simplicity, $\bm{H}^s(\Omega)$(respectively $\bm{H}_0^s(\Omega)$) stands for vector space $(H^s(\Omega))^2$(respectively $(H^s_0(\Omega))^2$). Besides, we will use a few times of homogeneous Sobolev space, which is stated as
\begin{align*}
& \dot{\bm{H}}(\Omega) := \{ \bm{u} \in (L_{loc}^2(\Omega))^2; \int_\Omega|\nabla \bm{u}|^2\dif x<\infty\}.
\end{align*}

The vector space $\mathcal{D}$ is made up of divergence-free vector fields in $(C_0^\infty(\Omega))^2$. Similarly, the symbol $V$ denotes the subset of $\bm{H}_0^1(\Omega)$, in which vector fields are all divergence-free. Moreover, the function space $L^2_{\sigma}$ is defined as
\begin{align*}\label{2}
  & L^2_{\sigma} := \lbrace \textbf{u} \in (L^2(\Omega))^2; \ \text{div}\, \textbf{u} = 0, \textbf{u} \cdot \bm{\nu}|_{\Gamma} = 0 \rbrace,
\end{align*}
here $\bm{\nu}$ is the normal vector to $\Gamma$.

For a scale function $\psi$, we denote $(-\partial_2 \psi, \partial_1 \psi)$ as $\nabla^\perp \psi$, while for a vector field $\bm{u}$, we will use notation $\nabla^\perp \cdot \bm{u} := -\partial_2 \bm{u}_1 + \partial_1 \bm{u}_2$.

Let $L > 0$ be arbitrary, we set $\Omega_{L} := \Omega \cap B(0, L)$ and $\Omega^L := \Omega \setminus \overline{B(0, L)}$, where $B(0, L)$ is the disk centered at origin with radius $L$. By the way, the unit disk centered at origin in $\mathbb{R}^2$ is denoted by $D$. Let $\mathcal{A}$  be an arbitrary set of $\mathbb{R}^2$, $\delta(\mathcal{A})$ represents the $\textit{diameter}$ of $\mathcal{A}$, that is
\begin{align}
\nonumber \delta(\mathcal{A}) := \sup_{x, y \in \mathcal{A}} |x - y|.
\end{align}

In the sequel of this section, firstly, we will elaborate why we need to construct approximate equations for \eqref{second-grade-rewriten-1}--\eqref{second-grade-rewriten-5}, then we will state our main results.

Observe that equation \eqref{vstream-1} of the vorticity stream formula is a linear transport equation provided $\bm{u}$ is a known flow velocity field. To apply the method used in \cite{2109.00915}, we need the property that $q$ is compactly supported when the initial data $q_0$ is compactly supported. However, this does not holds for equation \eqref{vstream-1}. In fact, from \eqref{vstream-1}, we formally have that
\begin{equation*}
{q}(\bm{X}(t, \alpha), t) = q_0(\alpha) - \nu \int_0^t q(\bm{X}(\alpha,s), s)\dif s + \nu \int_0^t \nabla^\perp\cdot( \bm{u}(\bm{X}(\alpha, s),s)) \dif s,
\end{equation*}
where $X(\cdot, t): \alpha\in\Omega \rightarrow X(\alpha, t)\in \Omega$ is the particle-trajectory mapping corresponding to flow velocity $\bm{u}$.
Observing that $\bm{u}$ is not compactly supported in general, it follows that $q(\cdot, t)$ would not be compactly supported for $t > 0$. And that is why we need to construct approximate equations for the second-grade fluid equations \eqref{second-grade-rewriten-1}--\eqref{second-grade-rewriten-5}. Let $\vartheta \in C_0^\infty(\mathbb{R}^2)$ such that $0 \leqslant \vartheta \leqslant 1$ in $\mathbb{R}^2$, $\vartheta(x) \equiv 1$ for $|x| < \frac{1}{2}$ and $\vartheta(x) \equiv 0$ for $|x| > 1$. Set $\vartheta_n(x) := \vartheta(\frac{x}{n})$ for $x \in \mathbb{R}^2$ with $n \in \mathbb{Z}^+$. The approximate equations are defined as:
\begin{numcases}{}
  \partial_t \bm{v^n} + \bm{u^n} \cdot {\nabla} \bm{v^n} + (\nabla \bm{u^n})^t \cdot \bm{v^n} + \nu \bm{v^n} - \nu \vartheta_{n} \bm{u^n} + \nabla p = 0   &  $\text{in}  \ \Omega \times (0, T), $ \label{second-approximate-1}  \qquad \    \\
    \text{div} \ \bm{u^n} = 0 &  $\text{in} \ \Omega \times [0, T) $ \label{second-approximate-2},\\
    \bm{u^n} = 0  & $\text{on} \ \Gamma \times [0, T) \label{second-approximate-3}$,\\
    \bm{u^n}|_{t=0} = \bm{u}_0^n & $\text{in} \ \Omega $\label{second-approximate-4},\\
    \bm{u^n} (x, t) \to 0 & $\forall t \in [0, T), |x| \to \infty $ \label{second-approximate-5},
\end{numcases}
where $\bm{v}^n = \bm{u}^n - \Delta \bm{u}^n$, and $\bm{u}^n_0$ is initial data. For the above initial boundary value problem, we have the following proposition.
\begin{prop} \label{proposition-1} Let $T > 0$ be fixed. Suppose the initial velocity $\bm{u}_0^n \in \bm{H}^s(\Omega) \cap V$($s \geq 3$) is compactly supported. Then the approximate equations $(\ref{second-approximate-1})$--$(\ref{second-approximate-5})$ has a unique weak solution $\textbf{u} \in L^\infty([0, T);\bm{H}^s(\Omega)\cap V) \cap C([0, T]; \bm{H}^3(\Omega))$ in the following sense: for any $\varphi \in C_0^\infty([0, T); \mathcal{D})$, the identity
\begin{equation}\label{energy-formula-approximate}
\begin{aligned}
   \int_0^T(\bm{v^n}, \partial_s\varphi)_{L^2(\Omega)} \dif s  &= (\bm{u}^n_0, \varphi(\cdot, 0))_{L^2(\Omega)} + (\nabla \bm{u}^n_0, \nabla \varphi(\cdot, 0))_{L^2(\Omega)}  \\
&\quad+   \nu\int_0^T ((1-\vartheta_n)\bm{u}^n, \varphi)_{L^2(\Omega)^2} \dif s + \nu \int_0^T (\nabla \bm{u}^n, \nabla \varphi)_{L^2(\Omega)} \dif s\\
&\quad+ \int_0^T (\bm{u}^n\cdot \nabla \bm{v}^n + (\nabla \bm{u}^n)^t \cdot \bm{v}^n, \varphi)_{L^2(\Omega)} \dif s
\end{aligned}
\end{equation}
holds.
\end{prop}
\begin{remark} The cut-function $\vartheta_n$ in the approximate equation \eqref{second-approximate-1}  helps to confine the support of \textit{unfiltered} vorticity $q^n$($\nabla^\perp\cdot \nu^n$) to a compact subset of $\Omega$, which will be verified in Lemma $\ref{lemma-transport}$. It is this property that we can show local well-posedness of the approximate equations.
\end{remark}

With proposition \ref{proposition-1} at hand, using compactness arguments, we may establish the following main theorem.
\begin{theorem}\label{theorem-1} For arbitrarily fixed $T > 0$. Suppose the initial data $\bm{u}_0 \in \bm{H}^s(\Omega) \cap V$($s \geq 3$), then the second-grade fluid equations $(\ref{second-grade-rewriten-1})$--$(\ref{second-grade-rewriten-5})$ has a unique weak solution $\bm{u} \in L^\infty([0, T); \bm{H}^s(\Omega)) \cap C([0, T); V)$ in the following sense: for any $\varphi \in C_0^\infty([0, T); \mathcal{D})$, the identity
\begin{equation}\label{energy-formula}
\begin{aligned}
   \int_0^T(\bm{v}, \partial_s\varphi)_{L^2(\Omega)} \dif s  &= (\bm{u}_0, \varphi(\cdot, 0))_{L^2(\Omega)} + (\nabla \bm{u}_0, \nabla \varphi(\cdot, 0))_{L^2(\Omega)}  \\
&\quad + \nu \int_0^T (\nabla \bm{u}, \nabla \varphi)_{L^2(\Omega)} \dif s + \int_0^T (\bm{u}\cdot \nabla \bm{v} + (\nabla \bm{u})^t \cdot \bm{v}, \varphi)_{L^2(\Omega)} \dif s
\end{aligned}
\end{equation}
holds. Moreover,
\begin{align}\label{theorem-bound}
\sup_{t\in[0, T]}\|\bm{u}(t)\|_{\bm{H}^s(\Omega)} \leqslant C(T, \|\bm{u}_0\|_{\bm{H}^3(\Omega)}) \|\bm{u}_0\|_{\bm{H}^s(\Omega)},
\end{align}
where $C(T, \|\bm{u}_0\|_{\bm{H}^3(\Omega)}) \equiv C$ is a constant depends on $T$ and $\|\bm{u}_0\|_{\bm{H}^3(\Omega)}$ for $s > 3$, while only depends on $T$ for $s = 3$.
\end{theorem}

\section{Preliminaries and technical lemmas}

In this section, we will give some preliminary results and a few technical lemmas. We observe that the initial velocity $\bm{u}_0^n$ of the approximate equations \eqref{second-approximate-1}--\eqref{second-approximate-5} should be compactly supported and converges to $\bm{u}_0$ in $\bm{H}^s(\Omega)$($s \geq 3$) as $n \rightarrow \infty$. Constructing such a approximate sequence for $\bm{u}_0$ is non-trivial, because the approximate vector $\bm{u}^n_0$ should be solenoidal and vanishes on boundary $\Gamma$. Fortunately, we have established the following result in our previous work \cite{2109.00915}.
\begin{lemma}\label{approximate-lemma}Let $s \in \mathbb{N}^+$ be fixed.  Suppose that $\bm{u}_0 \in V \cap \bm{H}^s(\Omega)$, then there exists a family of approximations $\{\bm{u}^n_0\} \subset V \cap \bm{H}^s(\Omega)$ such that $\bm{u}^n_0$ is compactly supported and converges to $\bm{u}_0$ in $\bm{H}^s(\Omega)$ strongly as $n\rightarrow \infty$.
\end{lemma}

As we have mentioned that each equation in the voriticity-stream formula \eqref{vstream-1}--\eqref{vstream-3} is linear and classic. Next, the existence, uniqueness and estimates of solution to these questions will be presented in detail. Since the equation \eqref{vstream-1} does not holds the special property that stated in section 2, we consider instead the approximate one
\begin{numcases}{}
    \partial_t q + \textbf{u} \cdot \nabla{q} + \nu q -\nu \nabla^\perp \cdot (\vartheta \bm{u})= 0, \label{transport-1}\\
    q|_{t=0} = q_0 \label{transport-2}.
\end{numcases}
where $q_0$ is the initial \textit{unfiltered} vorticity. For the above transport equations, we have the following result.
\begin{lemma}\label{lemma-transport}{}
Let $T > 0$ be fixed. Suppose $q_0 \in L^2(\Omega), \textbf{u}\in L^\infty([0, T]; V)$. Let $\vartheta \in C_0^\infty(\mathbb{R}^2)$ such that
$ 0 \leqslant \vartheta \leqslant 1$ in $\mathbb{R}^2$. Then the transport equations \eqref{transport-1}--\eqref{transport-2} have a unique weak solution  q $\in C([0, T]; L^2(\Omega))$. Moreover,
\begin{itemize}
\item[(i)] For $p =1, 2$, the following estimate holds:
\begin{equation}\label{transport-3}
     \begin{split}
        \|q(t)\|_{L^p({\Omega})}^p \leqslant \|q_0\|_{L^p({\Omega})}^p + C\nu\int_0^{t}\|\bm{u}(\cdot, s)\|_{\bm{H}^1(\Omega)}^p
     \end{split}
\end{equation}
for $t \in [0, T]$, where $C$ is a constant depends on $\Omega$ and the support of $\vartheta$ for $p = 1$, while depends only on $\Omega$ for $p=2$.
\item[(ii)]Suppose additionally that $\delta(\supp \, \vartheta) < R$ and $\delta(\supp \, q_0) < R$ for some $R > 0$, and $\bm{u} \in L^\infty([0, T]; \bm{H}^2(\Omega))$, then we have for all $t \in [0, T]$ that
\begin{equation} \label{remark-equation-1}
  \delta( \supp \, q(t)) \leqslant R + C\int_0^t \|\bm{u}(\cdot, s)\|_{\bm{H}^2(\Omega)} \dif s,
\end{equation}
where $C$ is a constant depends only on $\Omega$.
\item[(iii)] Suppose additionally that $q_0 \in \bm{H}^{s}(\Omega)$ and $\bm{u} \in L^\infty([0, T]; \bm{H}^{s+2}(\Omega))$, where $s \geq 1$, then it follows
\begin{equation} \label{high-transport-estimate}
\|q\|_{{H}^{s}(\Omega)} \leqslant C\left(\|q_0\|_{H^s(\Omega)} + \nu^\frac{1}{2}T^\frac{1}{2} \|\bm{u}\|_{L^\infty([0, T];\bm{H}^s(\Omega))}\right) \exp^{Ct\|\bm{u}\|_{L^\infty([0, T]; \bm{H}^{s+2}(\Omega))}},
\end{equation}
where $C$ is a constant depends only on $\Omega$.
\end{itemize}
\end{lemma}

\begin{proof}
We begin to prove a priori estimate $(\ref{transport-3})$. Fix $t \in [0, T]$, setting $\Omega^+ = \{x \in \Omega \ \big| \ q(t, x) > 0\}$, then $(\ref{transport-1})$ implies
\begin{align}
\partial_t q + \bm{u} \cdot \nabla q + \nu q \leqslant \nu |\nabla^\perp\cdot(\vartheta \bm{u})| \text{ in }\Omega^+.
\end{align}
Similarily setting $\Omega^- = \{x \in \Omega \ \big| \ q(t, x) < 0\}$, we have
\begin{align}
\partial_t (-q) + \bm{u} \cdot \nabla (-q) + \nu (-q) \leqslant \nu |\nabla^\perp\cdot(\vartheta \bm{u})| \text{ in }\Omega^-.
\end{align}
Collecting the above two inequalities, it follows that
\begin{align} \label{transport-tmp-3}
\partial_t |q| + \bm{u} \cdot \nabla |q| + \nu |q| \leqslant \nu |\nabla^\perp\cdot(\vartheta \bm{u})| \text{ in } \Omega.
\end{align}
Next, multiply ($\ref{transport-tmp-3}$) by $|q|^{p-1}$ and integrate over $\Omega \times [0, t]$, we find that
\begin{align}
\frac{1}{p}\|q\|_{L^p(\Omega)}^p + \nu \int_0^t \|q\|_{L^p}^p \leqslant \frac{1}{p}\|q_0\|_{L^p(\Omega)}^p + \nu \int_0^t\int_\Omega |\nabla^\perp\cdot(\vartheta \bm{u})| |q|^{p-1} \dif x \dif s.
\end{align}
By H\"older inequality and Young inequality,  the above inequality yields
\begin{align} \label{transport-tmp-4}
\|q\|_{L^p(\Omega)}^p \leqslant \|q_0\|_{L^p(\Omega)}^p + C\nu\int_0^t \|\nabla^\perp\cdot(\vartheta \bm{u})\|^p_{L^p(\Omega)} \dif s
\end{align}
Notice that $\vartheta$ is compactly supported, therefore
\begin{align}
\|\nabla^\perp\cdot(\vartheta \bm{u})\|^p_{L^p(\Omega)} \leqslant C\|\bm{u}\|^p_{\bm{H}^1(\Omega)},
\end{align}
then \eqref{transport-3} is followed by substitution of the above inequality into \eqref{transport-tmp-4}.

We now prove existence and uniqueness. Let us assume $\bm{u} \in C([0, T]; V)$ temporarily. Set $\bm{u}$ and $q_0$ to zero in $\mathbb{R}^2 \setminus \Omega$, and consider the following approximate equations in full plane $\mathbb{R}^2$:
\begin{numcases}{}
    \partial_t q^\varepsilon + J_{\varepsilon}\left[\bm{u} \cdot (\nabla J_{\varepsilon}q^\varepsilon)\right] + \nu q^\varepsilon -\nu \nabla^\perp\cdot(\vartheta \bm{u})= 0 & $ \text{ in } [0, T] \times \mathbb{R}^2$, \label{app-transport-1}\\
    q^\varepsilon|_{t=0} = q_0 & $\text { in } \mathbb{R}^2$\label{app-transport-2},
\end{numcases}
where $J_\varepsilon$ is the standard Friedrichs mollifier. Let us define map $\mathcal{G}: L^2(\mathbb{R}^2) \times [0, T] \rightarrow L^2(\mathbb{R}^2)$ as
\begin{align}
\mathcal{G}(q^\varepsilon, t) := J_{\varepsilon}\left[\bm{u} \cdot (\nabla J_{\varepsilon}q^\varepsilon)\right] + \nu q^\varepsilon -\nu \nabla^\perp\cdot(\vartheta \bm{u}).
\end{align}
It's easy to check that $\mathcal{G}$ is uniformly Lipschitz continuous in $q^\varepsilon$ for $t\in[0, T]$ and continuous in t. Indeed, let $q_1, q_2 \in L^2(\Omega)$, it follows
\begin{align}
\nonumber\|\mathcal{G}(q_1, t) - \mathcal{G}(q_2, t)\|_{L^2(\Omega)} &= \|\bm{u}(t)\cdot \nabla [J_{\varepsilon} (q_1 - q_2)] + \nu(q_1 - q_2)\|_{L^2(\Omega)}\\
\nonumber&\leqslant \|\bm{u}(t)\|_{L^2(\Omega)}\|\nabla [J_\varepsilon(q_1 - q_2)]\|_{L^\infty(\Omega)} + \nu \|q_1 - q_2\|_{L^2(\Omega)}\\
&\leqslant C\varepsilon^{-3}\|\bm{u}\|_{L^\infty([0, T];L^2(\Omega))} \|q_1 - q_2\|_{L^2(\Omega)} + \nu \|q_1 - q_2\|_{L^2(\Omega)},
\end{align}
$\mathcal{G}$ is therefore uniformly Lipschitz continuous in $q^\varepsilon$ for $t \in [0, T]$. On the other hand, for $t, s \in [0, T]$, we have
\begin{align}
\nonumber \|\mathcal{G}(q^\varepsilon, t) - \mathcal{G}(q^\varepsilon, s)\|_{L^2(\Omega)} &= \|(\bm{u}(t) - \bm{u}(s))\cdot \nabla [J_\varepsilon q^\varepsilon]\|_{L^2(\Omega)} + \nu\|\nabla [\vartheta(\bm{u}(t)-\bm{u}(s))]\|_{L^2(\Omega)}\\
&\leqslant C(\varepsilon^{-3}\|q^\varepsilon\|_{L^2(\Omega)} + \nu)\|(\bm{u}(t) - \bm{u}(s))\|_{\bm{H}^1(\Omega)}.
\end{align}
 Since $\bm{u}$ has be temporarily assumed to belong to $\in C([0, T];V)$, we conclude that $\mathcal{G}$ is continuous in $t$. By Cauchy-Lipschitz Theorem, there must exists a unique local solution $q^\varepsilon \in C^1([0, h]; L^2(\mathbb{R}^2))$ with some $h > 0$. From $(\ref{transport-3})$, we know that $\bm{u}$ is uniformly bounded, therefore the solution can extends to $[0, T]$.  Moreover, Banach-Alaoglu theorem implies that there exists a subsequence of  $\{q^\varepsilon\}$ converges weak-star to some $q$ in $L^\infty([0, T]; L^2(\Omega))$ as $\varepsilon \rightarrow 0$. It's easy to see that $q$ satisfies equation $(\ref{transport-1})$--$(\ref{transport-2})$ in weak sense.

We now consider the case $\bm{u} \in L^\infty([0, T]; V)$. Since $C([0, T]; V)$ is dense in $L^\infty([0, T]; V)$, there exists a sequence $\bm{u}^n \in C([0, T]; V)$ converges in $L^\infty([0, T]; V)$ to $\bm{u}$. We know for each $n \in \mathbb{N}$, the following equation
\begin{numcases}{}
    \partial_t q^n + \textbf{u}^n \cdot \nabla{q^n} + \nu q^n -\nu \nabla^\perp \cdot (\vartheta \bm{u}^n)= 0, \label{approximate-transport-1}\\
    q^n|_{t=0} = q_0 \label{approximate-transport-2}
\end{numcases}
has a unique solution $q^n \in L^\infty([0, T]; L^2(\Omega))$ satisfying \eqref{transport-3}, again by  Banach-Alaoglu theorem we know there exists a subsequence of $q^n$ converges  weak-star to some $q$ in $L^\infty([0, T]; V)$ as $n \rightarrow \infty$. It is easy to check that $q$ satisfies equation $(\ref{transport-1})$--$(\ref{transport-2})$ in weak sense.

We now show that $q \in C([0, T]; V)$. Indeed, multiplying $(\ref{app-transport-1})$ by $q$ and integrating over $\Omega$, it follows
\begin{align}
\frac{\dif}{\dif t} \|q\|_{L^2(\Omega)}^2 = - \nu \int_{\Omega} \left[\|q\|^2 - \nabla^\perp \cdot(\vartheta\bm{u}) q\right] \dif x.
\end{align}
Observing that the right hand of the above equality is integrable over $[0, T]$, $q$ is therefore Lipschitz continuous in $[0, T]$.

We now begin to prove ($\ref{remark-equation-1}$). Indeed, let $\bm{X}\in Lip([0, T]; (C(\Omega))^2)$ be the unique solution of the following ordinary differential equations
\begin{numcases}{}
\nonumber \frac{\dif \bm{X}(t, \alpha)}{\dif t} = \bm{u}(\bm{X}(t, \alpha), t) & $\text{in } \Omega \times [0, T]$,\\
\nonumber \bm{X}(0, \alpha) = \alpha & $ \forall \alpha \, \text{ in } \Omega$.
\end{numcases}
We observe that
\begin{equation*}
{q}(\bm{X}(t, \alpha), t) = q_0(\alpha) - \nu \int_0^t q(\bm{X}(\alpha,s), s)\dif s + \nu \int_0^t \nabla^\perp\cdot(\vartheta \bm{u}(\bm{X}(\alpha, s),s)) \dif s
\end{equation*}
for $(\alpha, t) \in \Omega \times [0, T]$. Let us consider those $\alpha \in \Omega \setminus B(0, R)$, recalling that $\delta(\supp \, \vartheta) < R$ and $\delta(\supp \, q_0) < R$, it follows that
\begin{align}
|q(\bm{X}(t, \alpha), t)| \leqslant \nu \int_0^s |q(\bm{X}(\alpha, s), s)| \dif s.
\end{align}
Then thanks to Gr\"onwall inequality, we know that
\begin{align} \label{lang-1}
q(\bm{X}(\alpha, t), t) = 0 \text{ for } (\alpha, t) \in \Omega \setminus B(0, R) \times [0, T].
\end{align}
Observe that
\begin{align}\label{lang-2}
\nonumber|\bm{X}(\alpha, t)| &\leqslant |\alpha| + \int_0^t \|\bm{u}(s)\|_{L^\infty(\Omega)}\dif s\\
&\leqslant |\alpha| + C\int_0^t \|\bm{u}(s)\|_{\bm{H^2}(\Omega)} \dif s.
\end{align}
Collecting \eqref{lang-1} and \eqref{lang-2}, we obtain $(\ref{remark-equation-1})$.

It remains to show (\ref{high-transport-estimate}). Indeed, differentiate equation (\ref{transport-1}) $\alpha$ times, we have
\begin{equation}
\nonumber \partial_t D^\alpha q + \textbf{u} \cdot \nabla{ (D^\alpha q)} + \sum_{\beta < \alpha, \beta + \gamma = \alpha} D^\gamma u \cdot \nabla (D^\beta q) + \nu D^\alpha q - \nu D^\alpha \nabla^\perp\cdot(\vartheta \bm{u})= 0
\end{equation}
Multiplying the above equation by $D^\alpha q$, integrating over $\Omega$ and summing up from $\alpha = 0$ to $|\alpha| = s$, we obtain a priori estimate:
\begin{equation}
\nonumber \frac{\dif}{\dif t} \|q(t)\|_{{H}^{s}(\Omega)}^2 \leqslant C (\|\bm{u}\|_{L^\infty([0, T];\bm{H}^{s+2}(\Omega))} \|q(t)\|_{H^{s}(\Omega)}^2 + C\nu\|\bm{u}\|_{L^\infty([0, T]; \bm{H}^{s+1}(\Omega))}^2)
\end{equation}
for $t \in [0, T]$, where we have used H\"older inequality.  by Gr\"onwall inequality, we conclude that (\ref{high-transport-estimate}) holds. The proof of this lemma is completed.
\end{proof}

Let us then consider equation \eqref{vstream-2} of the vorticity stream formula, which is the Poisson equation. As we know, there are many classical results to Poisson equation in a bounded domain with the Dirichlet boundary conditions:
\begin{numcases}{}
\Delta \psi = q & $ \text{in } \,\Omega$, \label{poisson-1} \\
\psi = 0 & $ \text{on } \, \Gamma$. \label{poisson-2}
\end{numcases}
For instance,  in \cite{gilbarg2015elliptic},the existence and uniqueness of solutions to the problem above was established. Furthermore, the authors proved that
\begin{align}\label{bdd-poisson-estimate}
 \|\psi\|_{\bm{H}^2(\Omega)} \leqslant C \|q\|_{L^2(\Omega)}.
\end{align}
In exterior domain, we could not get \eqref{bdd-poisson-estimate} in general, however the following results that established in our previous work \cite{2109.00915} will suffice in this article.
\begin{lemma}\label{lemma-poisson}{}
Let $R > 0$ be fixed. Suppose $q \in L^{2}(\Omega)$ and $\mathrm{supp}\, q \subset B(0, R)$, then the Poisson equations $(\ref{poisson-1})$--$(\ref{poisson-2})$
has a unique solution $\psi \in \dot{\bm{H}}(\Omega)$, which obeys the following estimates
\begin{align}
&\| \nabla \psi \|_{L^2(\Omega)} \leqslant CR(\|q\|_{L^2(\Omega)} + \|q\|_{L^1(\Omega)})\label{poisson-3-1}, \\
&\| D^2 \psi \|_{L^2(\Omega)} \leqslant C(R\|q\|_{L^2(\Omega)} + \|q\|_{L^1(\Omega)}) \label{poisson-3-2}.
\end{align}
Moreover, if $q \in H^s(\Omega)$ for $s \in \mathbb{N}$, it follows
\begin{equation}\label{esimate-du-m}
\begin{aligned}
\|D^{s+2}\psi\|_{L^2(\Omega)} \leqslant C ( \|q\|_{H^{s}(\Omega)} + \|\nabla \psi\|_{L^2(\Omega)}).
\end{aligned}
\end{equation}
\end{lemma}
\begin{remark} Althougth the $L^1(\Omega)$-norm of $q$ is bounded by its $L^2(\Omega)$-norm as $q$ is compactly supported, the $L^1(\Omega)$-norms of $q$ in \eqref{poisson-3-1} and \eqref{poisson-3-2} are reserved, for showing that the coefficients of the inequalities $(\ref{poisson-3-1})$ and $(\ref{poisson-3-2})$  depend linearly on $R$. This property would be a vital component in the proof of Proposition \ref{proposition-1}.
\end{remark}

Next, let us consider equation \eqref{vstream-3} of the vorticity stream formula, which is the stationary Stokes equation.  The following results can be found, for example in \cite{borchers1993boundedness, galdi2011introduction}.
\begin{lemma}\label{lemma-stokes}
Let $\varphi  \in L^2_{\sigma}(\Omega) \cap \bm{H}^1(\Omega)$, then the stationary Stokes equations:
\begin{numcases}{}
            \bm{u}(x) - \Delta \bm{u}(x) + \nabla p = \varphi(x) &  $\text{in } \Omega \label{stokes-1} $,\\
             \bm{u} = 0 & $\text{on } \Gamma$ \label{stokes-2}
\end{numcases}
has a unique solution $\displaystyle \bm{u} \in \bm{H}^3(\Omega) \cap V$, which obeys the estimate
\begin{equation}\label{stokes-3}
     \begin{split}
      \|\bm{u}\|_{\bm{H}^3(\Omega)} \leqslant C \|\varphi\|_{\bm{H}^1(\Omega)}.
     \end{split}
\end{equation}
\end{lemma}

At the end of this section, we introduce some further studies about Helmholtz–Weyl decomposition in exterior domain, which will be used to ensure that the vorticity stream formula \eqref{vstream-1}--\eqref{vstream-3} are equivalent to original equation \eqref{second-grade-rewriten-1}. As we know, in a simply connected domain, a smooth irrotational vector field is a grad field. Although two-dimensional exterior domain is not simple connected, Kozono and his collaborators has established similar results in \cite{hieber2021helmholtz}.
\begin{lemma}[\label{lemma-harmonic}see \cite{hieber2021helmholtz}] The following equations
\begin{numcases}{}
\nonumber \nabla \cdot \textbf{u} = 0 & $\text{in } \Omega $, \label{harmonic-1}\\
\nonumber \nabla^\perp \cdot \textbf{u} = 0 & $\text{in } \Omega $, \label{harmonic-2}\\
\nonumber \textbf{u} \cdot \nu = 0 & $\text{on } \Gamma$ \label{harmonic-3}
\end{numcases}
only has zero solution in $L^2(\Omega) \,space$.
\end{lemma}
\begin{cor}\label{remark-harmonic}
 Suppose that $\bm{u} \in (L^2(\Omega))^2$ is irrotational, then there exists a scalar function $p \in L^2_{loc}(\Omega)$ such that $\bm{u} = \nabla p$.
\end{cor}
\begin{proof}
 Indeed, from the classical Helmholtz decomposition, we have
\begin{align}
\bm{u} = \bm{v} + \nabla p
\end{align}
for some $\bm{v} \in L^2_{\sigma}$ and $p \in \dot{\bm{H}}(\Omega)$. Recalling that $\bm{u}$ is irrotational, it follows $\nabla^\perp \cdot \bm{v} = 0$, Lemma $\ref{lemma-harmonic}$ then implies $\bm{v} =0$.
\end{proof}

\section{Well-posedness of the approximate equations}
In this section, we will prove Proposition \ref{proposition-1}. For clarity's sake, we omit the superscript $n$  and subscript $n$  in \eqref{second-approximate-1}--\eqref{second-approximate-5} and rewrite them as:
\begin{numcases}{}
  \partial_t \bm{v} + \bm{u} \cdot {\nabla} \bm{v} + (\nabla \bm{u})^t \cdot \bm{v} + \nu \bm{v} - \nu \vartheta \bm{u} + \nabla p = 0   &  $\text{in}  \ \Omega \times (0, T), $ \label{second-approximate-1-rewritten}  \qquad \    \\
    \text{div} \ \bm{u} = 0 &  $\text{in} \ \Omega \times [0, T) $ \label{second-approximate-2-rewritten},\\
    \bm{u} = 0  & $\text{on} \ \Gamma \times [0, T] \label{second-approximate-3-rewritten}$,\\
    \bm{u}|_{t=0} = \bm{u}_0 & $\text{in} \ \Omega $\label{second-approximate-4-rewritten},\\
    \bm{u} (x, t) \to 0 & $\forall t \in [0, T), |x| \to \infty $ \label{second-approximate-5-rewritten}.
\end{numcases}

In view of Dirichlet boundary conditions, we could establish uniform estimate of $\bm{H}^1(\Omega)$ norm of $\bm{u}$ from the above equations directly. Indeed, multiplying $(\ref{second-approximate-1-rewritten})$ by $\bm{u}$ and integrating over $\Omega$, we obtain
\begin{align}
\nonumber\frac{1}{2}\frac{\dif} {\dif t}\|\bm{u}\|_{\bm{H}^1(\Omega)}^2 &= -\nu\|\nabla\bm{u}\|_{\bm{L}^2(\Omega)}^2 - \nu(1 - ||\vartheta||_{L^\infty})\|\bm{u}\|_{\bm{L}^2(\Omega)}^2 \leqslant 0,
\end{align}
which immediately yields
\begin{align}\label{bar_1_low_estimate}
\|\bm{u}(t)\|_{L^\infty([0, T]; \bm{H}^1(\Omega))} \leqslant \|\bm{u}_0^n\|_{\bm{H}^1(\Omega)}.
\end{align}

Observing that equations \eqref{second-approximate-1-rewritten}--\eqref{second-approximate-5-rewritten} are nonlinear and involve high order derivatives, it's not easy to obtain a priori estimate of high order derivatives of $\bm{u}$ directly. We therefore consider instead the following linearized vorticity stream formula:
\begin{numcases}{}
     \partial_t q + \bm{u} \cdot \nabla{q}  + \nu q - \nu \nabla^{\perp}\cdot (\vartheta\bm{u})= 0 & $ \text{in } \,\Omega \times [0, T]$ , \label{construction-equation-1} \\
     q|_{t=0} = q_0 & $ \text{in } \,  \Omega$,  \label{construction-equation-2}\\
     \Delta_x {\psi}(x, t) = {q}(x, t) & $ \text{in } \,\Omega \times [0, T]$,  \label{construction-equation-3}\\
     {\psi}(x, t) = 0 & $ \text{on } \, \Gamma \times [0, T]$ ,  \label{construction-equation-4}\\
      {\tilde{\bm{u}}}(x, t) + \bm{A}\tilde{\bm{u}}(x, t) = \nabla^{\perp}{\psi}(x, t) &  $\text{in } \Omega \times [0, T]$,   \label{construction-equation-5}\\
       \tilde{\bm{u}}(x, t) = 0 & $\text{on } \Gamma \times [0, T]$,   \label{construction-equation-6}
\end{numcases}
where $q_0 = \nabla^\perp \cdot (\bm{u}_0 - \Delta \bm{u}_0)$, and $\bm{A} = P(-\Delta)$ is the Stokes operator. We define mapping $\mathcal{F}:C([0, T_0]; V) \rightarrow C(0, T_0]; V)$ as $\mathcal{F}(\bm{u}) = \tilde{\bm{u}}$, the domain of which is given by
\begin{align}
\nonumber
 D(\mathcal{F}) := \{ \bm{u} \in C([0, T_0]; V \cap \bm{H}^3(\Omega) \ \big  | \ \|\bm{u}\|_{L^\infty([0, T_0]; \bm{H}^3(\Omega))} \leqslant M, \bm{u}|_{t=0} = \bm{u}_0\},
 \end{align}
where $ T_0 \in (0, T], M > 0$ are both determined later.
 Observe that the above equations becomes to be the vorticity stream formula of equations \eqref{second-approximate-1-rewritten}--\eqref{second-approximate-5-rewritten} provided that $\bm{u}$ is the fixed point of mapping $\mathcal{F}$, and as a result, $\bm{u}$ is exactly the solution to the approximate equations.

In the sequel of this section,  we will first show $\mathcal{F}$ is a contraction mapping, so we can use Banach fixed point theorem to show that there exists a fixed point of mapping $\mathcal{F}$. Next, we will prove the fixed point $\bm{u}$ is exactly the unique local solution to equations $(\ref{second-approximate-1-rewritten})$--$(\ref{second-approximate-5-rewritten})$. At last, the solution will be proved to be globally existent.

\begin{proof}[Proof of proposition $\ref{proposition-1}$]

$\bm{Step \ 1}$:  We begin to show that $\mathcal{F}$ is well-defined. For convenience, we will make use of the following norm
\begin{align}
\|\cdot \|_{L^{1}(\Omega)\cap L^2(\Omega)} := \|\cdot\|_{L^1(\Omega)} +\|\cdot\|_{L^2(\Omega)}.
\end{align}
Suppose $\bm{u}\in D(\mathcal{F})$, let us first consider equations \eqref{construction-equation-1}--\eqref{construction-equation-2}, which are linear transport equations. From Lemma $\ref{lemma-transport}$, we know that there exists a unique weak solution ${q} \in C([0, T_0]; L^1(\Omega) \cap L^2(\Omega))$, which admits the following estimate
\begin{equation}\label{pf-transport}
\| ({q(t)} \|_{L^1(\Omega)\cap L^2(\Omega)}) \leqslant C_{T}(\| q_0 \|_{L^1(\Omega) \cap L^2(\Omega)} + T_0^\frac{1}{2}\|\bm{u}\|_{L^\infty([0, T_0]; \bm{H}^1(\Omega))})
\end{equation}
for $\forall \ t \in [0, T_0]$, where $C_{T}$ is a constant depends on $T$, $\Omega$ and $\nu$. Moreover, observing that $\theta$ and $q_0$ are both supported in disk $B(0, R)$ for some $R >0$ large enough,  therefore property (ii) of Lemma \ref{lemma-transport} tells that for $t \in [0, T_0]$,
\begin{equation}\label{pf-transport-1}
\delta(\supp \, q(t))  < R + C \int_0^t \|\bm{u}(\cdot, s)\|_{\bm{H}^2(\Omega)} \dif s.
\end{equation}
We then consider equations (\ref{construction-equation-3})--(\ref{construction-equation-4}), which are Poisson equations for fixed $t\in[0, T_0]$. It follows from Lemma $\ref{lemma-poisson}$ that there exists a unique solution ${\psi} \in \dot{\bm{H}}(\Omega)$ satisfies the following property:
\begin{equation}\label{pf-poisson}
\| \nabla {\psi}(t) \|_{\bm{H}^1(\Omega)} \leqslant C\delta(\supp \, q(t)) (\|q(t)\|_{L^1(\Omega) \cap L^2(\Omega)}).
\end{equation}
Now let us move on to equations \eqref{construction-equation-5}--\eqref{construction-equation-6}. Since they are exactly the stationary Stokes equations for fixed $t \in[0, T_0]$, we know from Lemma $\ref{lemma-stokes}$ that the solution $\bm{\tilde{u}}$ uniquely exists and obeys
\begin{equation}\label{pf-tmp-stokes}
\begin{aligned}
\|\bm{\tilde{u}}(t)\|_{\bm{H}^3(\Omega)} &\leqslant C \|\nabla^\bot {\psi}(t)\|_{\bm{H}^1(\Omega)}.
\end{aligned}
\end{equation}
Collecting $(\ref{pf-transport})$--$(\ref{pf-tmp-stokes})$ gives that for arbitrary $t\in[0, T_0]$, the following estimate
\begin{equation}\label{pf-stokes}
\begin{aligned}
\|\tilde{\bm{u}}(t)\|_{\bm{H}^3(\Omega)} \leqslant C_{T} (R + \int_0^t \|&\bm{u}(\cdot, s)\|_{\bm{H}^2(\Omega)} \dif s)(\|q_0\|_{L^1(\Omega)\cap L^2(\Omega)}+ T_0^\frac{1}{2}\|\bm{u}\|_{L^\infty([0, T_0]; \bm{H}^1(\Omega))})
\end{aligned}
\end{equation}
holds.  It follows that $\tilde{\bm{u}}$ belongs to $C([0, T]; \bm{H}^3(\Omega) \cap V)$, and as a result, mapping $\mathcal{F}$ is well-defined.

$\bm{Step \ 2}$: As we have mentioned that we anticipate to use Banach fixed point theorem to obtain local well-posedness, we shall now determine the parameters $M$ and $T_0$ to guarantee that $\bm{\tilde{u}} \in D(\mathcal{F})$. Indeed, setting $M := \max\left(4C_{T} R\| q_0 \|_{L^1(\Omega)\cap L^2(\Omega)}, \|\bm{u}_0\|_{\bm{H}^3(\Omega)}\right)$ and $T_0 = \min \left\{\frac{R}{M}, \frac{\|q_0^n\|_{L^1(\Omega) \cap L^2(\Omega)}^2}{M^2}\right\}$, then ($\ref{pf-stokes}$) implies
\begin{equation}\label{pf-tilde-u-estimate}
 \sup_{t\in[0, T_0]}\|\bm{\tilde{u}}(t)\|_{\bm{H}^3(\Omega)} \leqslant  M.
\end{equation}
In view of the domain of mapping $\mathcal{F}$, it remains to show that $\bm{\tilde{u}}|_{t=0} = \bm{u}_0$. Indeed, From equations (\ref{construction-equation-1})$-$(\ref{construction-equation-6}), we deduce that
\begin{equation}
\begin{aligned}
\nabla^\perp \cdot (\bm{\tilde{u}} + A\bm{\tilde{u}})|_{t=0} = \Delta {\psi}|_{t=0} = q_0 = \nabla^\perp \cdot(\bm{u}_0 + A\bm{u}_0).
\end{aligned}
\end{equation}
Besides, observing that $\bm{u}$ and $\bm{\tilde{u}}$ are both divergence-free, we immediately obtain
\begin{equation}
\begin{aligned}
\nabla \cdot (\bm{\tilde{u}} + A\bm{\tilde{u}})|_{t=0} = \nabla \cdot (\bm{u}_0 + A\bm{u}_0) = 0.
\end{aligned}
\end{equation}
Collecting the above two equalities, it follows from Lemma $\ref{lemma-harmonic}$ that
\begin{equation}
\begin{aligned}
(\bm{\tilde{u}}|_{t=0} -\bm{u}_0) + A(\bm{\tilde{u}}|_{t=0} - \bm{u}_0) = 0,
\end{aligned}
\end{equation}
which, in turn, yields $\bm{\tilde{u}}|_{t=0} = \bm{u}_0$ by Lemma \ref{lemma-stokes}. We thus have proved that $\tilde{\bm{u}} \in D(\mathcal{F})$.

$\bm{Step \ 3}$: We then prove that $\mathcal{F}$ is a contraction mapping. First of all, we assert that ($\bm{u}$, $\bm{\tilde{u}}$) satisfies the following equations
\begin{numcases}{}
    \partial_t \bm{\tilde{v}} + \bm{u} \cdot {\nabla} \bm{\tilde{v}} - (\nabla \bm{\tilde{v}})^t \bm{u} + \nu \tilde{\bm{v}} - \nu \vartheta\bm{u} + \nabla \tilde{p} = 0  &  $\text{in} \ \Omega \times (0, T_0] $ \label{pf-euler-alpha-1}\\
    \text{div} \ \bm{\tilde{u}} = 0, &  $\text{in} \ \Omega \times [0, T_0] $ \label{pf-euler-alpha-2}\\
    \bm{\tilde{u}} = 0,  & $\text{on} \ \Gamma \times [0, T_0] \label{pf-euler-alpha-3}$\\
    \bm{\tilde{u}}|_{t=0} = \bm{u}_0, & $\text{in} \ \Omega $\label{pf-euler-alpha-4}\\
    \bm{\tilde{u}} (x, t) \to 0, & $\forall t \in [0, T_0], |x| \to \infty $ \label{pf-euler-alpha-5}
\end{numcases}
where $\bm{\tilde{v}} = \bm{\tilde{u}} - \Delta \bm{\tilde{u}}$. We only need to check ($\ref{pf-euler-alpha-1}$). To this end, we define
\begin{equation}
\begin{aligned}
\nonumber \bm{Q}(t) :=  \bm{\tilde{v}}(t) -   \bm{\tilde{v}}(0) + \int_0^t \left[\bm{u}(s) \cdot {\nabla} \bm{\tilde{v}}(s)  - (\nabla \bm{\tilde{v}}(s))^t \bm{u}(s)  + \nu \tilde{\bm{v}}(s) - \nu \vartheta \bm{u}(s)\right] \dif s
\end{aligned}{}
\end{equation}
with $t \in [0, T_0]$. Observing that $\bm{u}, \bm{\tilde{u}} \in C([0, T_0]; \bm{H}^3(\Omega)\cap V)$, we have $\bm{Q}(t) \in C([0, T_0]; (L^2(\Omega))^2)$. Moreover,
\begin{equation}
\begin{aligned}
\nonumber \nabla ^\perp \cdot \bm{Q}(t) &= \nabla_x^\perp \cdot \big[\bm{\tilde{v}}(t) -   \bm{\tilde{v}}(0) + \int_0^t \left[\bm{u}(s) \cdot {\nabla} \bm{\tilde{v}}(s)  - (\nabla \bm{\tilde{v}}(s))^t \bm{u}(s)  + \nu \tilde{\bm{v}} (s)- \nu \vartheta_{n} \bm{u}(s)\right] \dif s\big] \\
&=\int_0^t \left[ \partial_s {q}(s) + \bm{u} \cdot \nabla_x{{q}(s)} + \nu q(s)  - \nu \nabla^{\perp}\cdot (\vartheta_{n}\bm{u}(s))\right] \dif s \\
&= 0,
\end{aligned}
\end{equation}
where we have used identity $(\ref{construction-equation-1})$. In view of Corollary \ref{remark-harmonic}, it follows that there exists a scalar function $\tilde{P} \in C([0, T_0];\bm{\dot{H}}(\Omega))$ such that
\begin{equation}
\begin{aligned}\label{q-equation}
 \bm{Q} = \nabla_x \tilde{P}(t).
\end{aligned}
\end{equation}
 By differentiating the above identity in t and setting $\tilde{p}(t, x):= \partial_t \tilde{P}(t, x)$, we immediately obtains \eqref{pf-euler-alpha-1}.

We now begin energy estimating via equations $(\ref{pf-euler-alpha-1})$--$(\ref{pf-euler-alpha-5})$. Suppose $\bm{u^1}, \bm{u^2} \in C([0, T_0]; V\cap\bm{H}^3(\Omega))$, we denote
\begin{equation}
\begin{aligned}
\bm{\tilde{u}}^1 = \cal{F}[{\bm{u^1}}],\\
 \bm{\tilde{u}}^2 = \cal{F}[\bm{u}^2],\\
 \bm{\tilde{v}}^1 = \bm{\tilde{v}}^1 -\Delta\bm{\tilde{u}}^1\\ \bm{\tilde{v}}^2 = \bm{\tilde{u}}^1 - \Delta\bm{\tilde{u}}^2,\\
W = \bm{u^1} - \bm{u^2},\\
S = \bm{\tilde{u}}^1 - \bm{\tilde{u}}^2.
\end{aligned}
\end{equation}
It is easy to see that
\begin{align}
\nonumber \bm{\tilde{v}}^1 - \bm{\tilde{v}}^2 = S - \Delta S.
\end{align}
Subtracting  equation ($\ref{pf-euler-alpha-1})$ for $\bm{\tilde{u}}^2$ from that for $\bm{\tilde{u}}^1$, it follows
\begin{equation}
\begin{aligned}
\partial_t(S-\Delta S) &= -\bm{u^1} \cdot \nabla \bm{\tilde{v}}^1 + \bm{u^2} \cdot \nabla \bm{\tilde{v}}^2 - \sum_{j=1}^2 \bm{u^1}_j \cdot \nabla \bm{\tilde{v}}^1_j + \sum_{j=1}^2 \bm{u^2}_j \cdot \nabla \bm{\tilde{v}}^2_j \\
&\quad \,- \nabla (\tilde{p}^1 - \tilde{p}^2) -\nu(\tilde{\bm{v}}^1 -\tilde{\bm{v}}^2) + \nu \vartheta(\bm{u}^1 - \bm{u}^2).
\end{aligned}
\end{equation}
Multiplying the above equation with $S$ and integrating over $\Omega \times [0, t)$ for arbitrary $t \in (0, T_0]$, we obtain
\begin{equation}\label{pf-K}
\begin{aligned}
\frac{1}{2}\left(\|S(t)\|_{L^2}^2 + \|\nabla S(t)\|_{L^2}^2\right) &= -\int_0^t\int_{\Omega} S \cdot \left[ (W\cdot \nabla) \bm{\tilde{v}}^1 + (\bm{u^2}\cdot \nabla) (S - \Delta S)\right] \dif x \dif s\\
&\quad-\int_0^t\int_{\Omega} S \cdot \left[ \sum_{j=1}^2 \bm{u}_j^1\nabla (S_j - \Delta S_j) + \sum_{j=1}^2 W_j \nabla \bm{\tilde{v}}_j^2\right] \dif x \dif s\\
&\quad-\nu\int_0^t\int_{\Omega} S \cdot [(\tilde{\bm{v}}^1 - \tilde{\bm{v}}^2) - \vartheta(\bm{u}^1 - \bm{u}^2)]\dif x \dif s\\
&=:K_1 + K_2 + K_3,
\end{aligned}
\end{equation}
where we have used the property that $\tilde{\bm{u}}^2|_{t = 0} = \tilde{\bm{u}}^1_{t = 0} = \bm{u}_0^n$. We will examine each term in $(\ref{pf-K})$. We begin by estimating $K_1$. We note, as usual, that $(S, \bm{u^2} \cdot \nabla S) = 0$, using integration by parts, we deduce
\begin{align}
\nonumber K_1 &= -\int_0^t\int_{\Omega}S \cdot \left[ (W\cdot \nabla) \bm{\tilde{v}}^1 + (\bm{u^2}\cdot \nabla) (S - \Delta S)\right] \dif x \dif s\\
\nonumber &= -\int_0^t\int_{\Omega}S \cdot \left[ (W\cdot \nabla) \bm{\tilde{v}}^1\right] \dif x \dif s +\int_0^t\int_{\Omega}S \cdot \left[ (\bm{u^2}\cdot \nabla) (\Delta S)\right] \dif x \dif s \\
\nonumber &= -\int_0^t\int_{\Omega}S \cdot \left[ (W\cdot \nabla) \bm{\tilde{v}}^1\right] \dif x \dif s -\int_0^t\int_{\Omega}\Delta S \cdot \left[ (\bm{u^2}\cdot \nabla) S\right] \dif x \dif s\\
\nonumber&=  -\int_0^t\int_{\Omega}S \cdot \left[ (W\cdot \nabla) \bm{\tilde{v}}^1 \right] \dif x \dif s  + \int_0^t\int_{\Omega}\sum_{k, l} \left[ \partial_l u_k^2 \partial_k S + u_k^2\partial_k\partial_l S\right] \partial_l S \dif x \dif s\\
\nonumber&= -\int_0^t\int_{\Omega}S \cdot \left[ (W\cdot \nabla) \bm{\tilde{v}}^1 \right] \dif x \dif s  + \int_0^t\int_{\Omega}\sum_{k, l} \left[ \partial_l u_k^2 \partial_k S\right] \partial_l S \dif x \dif s,
\end{align}
then taking account of ($\ref{pf-tilde-u-estimate}$) and using H\"older inequality and Gagliardo–Nirenberg interpolation inequality, it follows
\begin{align}
\nonumber |K_1| &\leqslant \int_0^t \left[\|S\|_{L^4} \|W\|_{L^4} \|\nabla \bm{\tilde{v}}^1\|_{L^2} +  \|\nabla \bm{u^2}\|_{L^\infty}\|\nabla S\|_{L^2}^2\right] \dif s\\
&\leqslant CM \int_0^t\left[\|S\|_{\bm{H}^1(\Omega)}^2 + \|W\|_{\bm{H}^1(\Omega)}^2\right] \dif s \label{pf-K_1}.
\end{align}
 We now estimate the second term $K_2$, which is similar to $K_1$. Indeed, by integration by parts, $K_2$ can be rewrited as
\begin{align}
\nonumber K_2 &= -\int_0^t\int_{\Omega} S \cdot \left[ \bm{u}_j^1\nabla (S_j -  \Delta S_j) +  W_j \nabla \bm{\tilde{v}}_j^2\right] \dif x \dif s \\
\nonumber &= -\int_0^t\int_{\Omega} \left[S_i \bm{u}_j^1\partial_i S_j -  S_i \bm{u}_j^1\partial_i(\Delta S_j) +  S_i W_j \partial_i \bm{\tilde{v}}_j^2 \right]\dif x \dif s\\
\nonumber &= -\int_0^t\int_{\Omega} \left[S_i \bm{u}_j^1\partial_i S_j  +  S_i \partial_i \bm{u}_j^1 \Delta S_j +  S_i W_j \partial_i \bm{\tilde{v}}_j^2 \right] \dif x \dif s \\
\nonumber &= -\int_0^t\int_{\Omega} \left[S_i\bm{u}_j^1 \partial_i S_j - \partial_k \partial_l \bm{u^1}_jS_l\partial_k S_j\right] \dif x \dif s \\
\nonumber &\quad+ \int_0^t\int_{\Omega}\left[- \partial_l \bm{u^1}_j\partial_kS _l\partial_k S_j +  \partial_i \bm{\tilde{v}}^2_j S_i W_j\right] \dif x \dif s,
\end{align}
again we use inequality ($\ref{pf-tilde-u-estimate}$), H\"older inequality and Gagliardo–Nirenberg interpolation inequality to deduce that
\begin{align}
\nonumber |K_2| &\leqslant \int_0^t \left[\|S\|_{L^2}\|\bm{u^1}\|_{L^\infty}\|\nabla S\|_{L^2} + \|\nabla^2\bm{u^1}\|_{L^4}\|S\|_{L^4}\|\nabla S\|_{L^2} \right] \dif s \\
\nonumber &\quad + \int_0^t \left[\|\nabla \bm{u^1}\|_{L^\infty}\|\nabla S\|_{L^2}^2  + \|S\|_{L^4}\|W\|_{L^4}\|\nabla \bm{\tilde{v}}^2\|_{L^2}\right] \dif s\\
&\leqslant CM \int_0^t \left[\|S\|_{\bm{H}^1(\Omega)}^2 + \|W\|_{\bm{H}^1(\Omega)}^2\right] \dif s \label{pf-K_2}.
\end{align}
We are left to estimate the last term $K_3$.  By integration by parts, we find that
\begin{align}
\nonumber K_3 &= -\nu\int_0^t\int_{\Omega} S \cdot [(\tilde{\bm{v}}^1 - \tilde{\bm{v}}^2) - \vartheta(\bm{u}^1 - \bm{u}^2)]\dif x \dif s\\
\nonumber &=- \nu \int_0^t \|S\|_{\bm{H}^1(\Omega)}^2 \dif s+ \nu \int_0^t \int_\Omega \vartheta S \cdot W \dif x \dif s,
\end{align}
using H\"older inequality and Young inequality, it follows that
\begin{align}
|K_3| &\leqslant  C\int_0^t  \|W\|_{\bm{H}^1(\Omega)}^2 \dif s \label{pf-K_3},
\end{align}
where $C$ depends on $\nu$. Collecting (\ref{pf-K})-- (\ref{pf-K_3}), we find that for all $t \in [0, T_0]$
\begin{equation}
\begin{aligned}
\|S(t)\|_{\bm{H}^1(\Omega)}^2 \leqslant CM \int_0^t \left[\|S\|_{\bm{H}^1(\Omega)}^2 + \|W\|_{\bm{H}^1(\Omega)}^2\right] \dif s,
\end{aligned}
\end{equation}
where M is defined in step 2 and $C$ only depends on $\Omega$, $\nu$. Thanks to Gr\"onwall inequality, we have
\begin{equation}
\begin{aligned}
\|S(t)\|_{\bm{H}^1(\Omega)}^2 \leqslant CMe^{CMt}\int_0^t\|W(s)\|_{\bm{H}^1(\Omega)}^2 \dif s.
\end{aligned}
\end{equation}
The above inequality implies for $0 < h \leqslant T_0$
\begin{equation}
\begin{aligned}\label{pf-estimate-s}
\sup_{t\in[0, h]} \|S(t)\|_{\bm{H}^1(\Omega)}^2 \leqslant CMe^{CMh}h\sup_{t\in[0, h]}\|W(t)\|_{\bm{H}^1(\Omega)}^2.
\end{aligned}
\end{equation}
Consequently, if we choose $h$ small enough, such that $CMe^{CMh}h < 1$, then it follows that $\cal{F}$ is a contraction mapping with respect to the $\bm{H}^1(\Omega)$-norm.

$\bm{Step \ 4}$: Local well-posedness on $[0, h]$. By Banach fixed point theorem in metric spaces, we could conclude there exists a unique fixed point $\bm{u}\in C([0, h]; V)$. This fixed point is also the limit of the fixed point iteration, with $\bm{u}^0 := \bm{u}_0$ and $\bm{u}^n := \cal{F}[\bm{u}^{n-1}]$. From $(\ref{pf-tilde-u-estimate})$, we know that $\bm{u}^n$ is uniformly bounded in $L^\infty([0, h];\bm{H}^3(\Omega))$, thanks to  Banach-Alaoglu theorem we have that there exists a subsequence $\bm{u}^{n_k}$ converges, weak-star to $\bm{u}$ in  $L^\infty([0, h]; \bm{H}^3(\Omega))$. It's easy to see that $\bm{u} \in C([0, h]; \bm{H}^3(\Omega))$. Indeed, since $\cal{F}[\bm{u}] \in C([0, h]; \bm{H}^3(\Omega)\cap V)$ and $\bm{u}$ is the fixed point, we have $\bm{u} := \cal{F}[\bm{u}] \in C([0, h]; \bm{H}^3(\Omega)\cap V)$.

Now we are going to show that the fixed point $\bm{u}$ satisfies equations \eqref{second-approximate-1-rewritten}--\eqref{second-approximate-5-rewritten}. Indeed, since $\bm{u}$ is the fixed point of $\mathcal{F}$, substituting $\tilde{\bm{u}}$ by $\bm{u}$ in the equations \eqref{pf-euler-alpha-1}--\eqref{pf-euler-alpha-5} immediately tells that $\bm{u}$  satisfies \eqref{second-approximate-1-rewritten}--\eqref{second-approximate-5-rewritten}.

$\bm{Step \ 5}$: Extend to [0, T].
Observing that $\bm{u}$ is the fixed point of $\mathcal{F}$, it follows from $(\ref{pf-stokes})$ and $(\ref{bar_1_low_estimate})$
\begin{align}
\|\bm{u}(\cdot, t)\|_{\bm{H}^3(\Omega)} \leqslant C(T, \|\bm{u}_0^n\|_{\bm{H}^3(\Omega)})(R + \int_0^t \|\bm{u}(\cdot, s)\|_{\bm{H}^3(\Omega)} \dif s).
\end{align}
Thanks to Gr\"onwall inequality, we arrive at
\begin{align}
\|\bm{u}(\cdot, t)\|_{\bm{H}^3(\Omega)} \leqslant C(T, \|\bm{u}_0^n\|_{\bm{H}^3(\Omega)}) R.
\end{align}
With the above estimate at hand, we can prove by contradiction that the solution exists on $[0, T]$. Since $T$ is arbitrary, the solution is global.
\end{proof}
\begin{remark} \label{remark-propostion-1}Let $s \geq 4$. Using the arguments in the construction process of $\mathcal{F}$ and combing with \eqref{high-transport-estimate}, one can show that $\bm{u} \in L^\infty([0, T]; \bm{H}^s(\Omega) \cap V)$ provided that $\bm{u}_0 \in \bm{H}^s(\Omega)$. Indeed, we observe that $\bm{u}$ is the fixed point of $\mathcal{F}$, Lemma $\ref{lemma-stokes}$ implies that
\begin{align}
\nonumber\|\bm{u}\|_{\bm{H}^s(\Omega)} \leqslant C \|\nabla \psi\|_{\bm{H}^{s-1}(\Omega)}.
\end{align}
Collecting with $\eqref{high-transport-estimate}$, $\eqref{esimate-du-m}$ and the above inequality, it follows
\begin{align}
\nonumber\|\bm{u}\|_{L^\infty([0, T]; \bm{H}^s(\Omega))} \leqslant C(\|\bm{u}\|_{L^\infty([0, T]; \bm{H}^{s-1}(\Omega))}, T)\|\bm{u}_0\|_{\bm{H}^s(\Omega)}.
\end{align}
By induction arguments, we have $\bm{u} \in L^\infty([0, T]; \bm{H}^s(\Omega))$.
\end{remark}

\section{Global well-posedness of second-grade equations}

Formally, let $n \rightarrow \infty$, the approximate equations \eqref{second-approximate-1}--\eqref{second-approximate-5} turn back to original equations \eqref{second-grade-rewriten-1}--\eqref{second-grade-rewriten-5}. In this section, the convergence would be verified by compactness arguments.
\begin{proof}[Proof of theorem $\ref{theorem-1}$]
Let $T > 0$ be arbitrary, and suppose $\bm{u}_0 \in V \cap \bm{H}^s(\Omega)$ for some $s \geq 3$. Let $\lbrace \bm{u}^n_0\rbrace$ be the approximate sequence constructed in Lemma $\ref{approximate-lemma}$ such that $\delta(\supp \, \bm{u}^n_0) < n$. For each fixed $n\in\mathbb{N}^+$, Proposition $\ref{proposition-1}$ and Remark \ref{remark-propostion-1} tell us there exists a unique solution $\bm{u}^n \in L^\infty([0, T]; \bm{H}^s(\Omega)) \cap C([0, T]; \bm{H}^3(\Omega) \cap V)$ to equations \eqref{second-approximate-1}--\eqref{second-approximate-5}. That is, identity \eqref{energy-formula-approximate} holds for arbitrary $\varphi \in C_0^\infty([0, T); \mathcal{D})$. We now prove $\bm{u}^n$ converges to an unique solution $\bm{u}$ of the second-grade fluid equations \eqref{second-grade-rewriten-1}--\eqref{second-grade-rewriten-5} as $n\rightarrow \infty$.

 Let us first show $\bm{u}^n$ is uniformly bounded in $L^\infty([0, T]; \bm{H}^s(\Omega))$. Since $\bm{u}^n$ is the fixed point of $\mathcal{F}$, we know from $(\ref{construction-equation-1})$--$(\ref{construction-equation-6})$ that there exists a pair $(q^n, \psi^n) \in C([0, T]; L^2(\Omega)) \times C([0, T]; \dot{\bm{H}}(\Omega))$  such that the following equations
\begin{numcases}{}
    \nonumber \partial_t q^n + \bm{u}^n \cdot \nabla{q^n} + \nu q^n -\nu \nabla^\perp \cdot (\vartheta_{n} \bm{u}^n) = 0 & $ \text{in } \,\Omega \times [0, T]$ , \\
    \nonumber q^n|_{t=0} = q_0^n & $ \text{in } \,  \Omega$, \\
    \nonumber \Delta_x {\psi^n}(x, t) = {q^n}(x, t) & $ \text{in } \,\Omega \times [0, T]$,\\
    \nonumber {\psi^n}(x, t) = 0 & $ \text{on } \, \Gamma \times [0, T]$, \\
    \nonumber {{\bm{u}}}^n(x, t) + \bm{A}{\bm{u}}^n(x, t) = \nabla^{\perp}{\psi^n}(x, t) &  $\text{in } \Omega \times [0, T]$,   \\
    \nonumber {\bm{u}}^n(x, t) = 0 & $\text{on } \Gamma \times [0, T]$
\end{numcases}
hold. We observe that for fixed parameter $t \in [0, T]$, $(\bm{u}^n, \nabla^\perp \psi^n)$ satisfies the stationary Stokes equations, it follows from Proposition 6 in \cite{giga2018handbook} that
\begin{equation}
\|D^{k+3}\bm{u}^n(\cdot, t)\|_{\bm{L}^2(\Omega)} \leqslant C(\|\nabla^\perp\psi^n(\cdot, t)\|_{\bm{H}^{k+1}(\Omega)} + \|\bm{u}^n(\cdot, t)\|_{\bm{L}^2(\Omega)}),
\end{equation}
where $0 \leqslant k \leqslant s$. Since $(\psi^n, q^n)$ satisfies Poisson equations, combing the estimate ($\ref{esimate-du-m}$) in Lemma \ref{lemma-poisson}  and the above inequality gives
\begin{align}
\|D^{k+3}\bm{u}^n(\cdot, t)\|_{\bm{L}^2(\Omega)} \leqslant C(\|q^n(\cdot, t)\|_{\bm{H}^k(\Omega)} + \|\nabla\psi^n(\cdot, t)\|_{\bm{L}^2(\Omega))} + \|\bm{u}^n(\cdot, t)\|_{\bm{L}^2(\Omega)}).
\end{align}
Noting that $\nabla^\perp\psi^n = \bm{u}^n + A\bm{u}^n$, therefore by Gagliardo–Nirenberg interpolation theorem, we arrive at
\begin{equation}\label{high-estimate-u}
\begin{aligned}
\|D^{k+3}\bm{u}^n(\cdot, t)\|_{\bm{L}^2(\Omega)} &\leqslant C(\|q^n(\cdot, t)\|_{\bm{H}^k(\Omega)} + \|\bm{u}^n(\cdot, t)\|_{\bm{L}^2(\Omega)}).
\end{aligned}
\end{equation}
Let us first consider the case $k = 0$. Collecting ($\ref{transport-3}$), ($\ref{bar_1_low_estimate}$) and \eqref{high-estimate-u} gives
\begin{equation}
\begin{aligned}
\|\bm{u}^n\|_{L^\infty([0, T];\bm{H}^3(\Omega))} &\leqslant C_T\|\bm{u}^n_0\|_{\bm{H}^3(\Omega)}\leqslant C_T\|\bm{u}_0\|_{\bm{H}^3(\Omega)}.
\end{aligned}
\end{equation}
Taking into account (\ref{high-transport-estimate}) and $(\ref{high-estimate-u})$, we deduce by induction that
\begin{equation}\label{ub-u}
\begin{aligned}
\|\bm{u}^n\|_{L^\infty([0, T]; \bm{H}^s(\Omega))} &\leqslant C(T, \|\bm{u}_0\|_{\bm{H}^3(\Omega)})\|\bm{u}_0\|_{\bm{H}^s(\Omega)},
\end{aligned}
\end{equation}
where $C(T, \|\bm{u}_0\|_{\bm{H}^3(\Omega)})$ is a constant depends on $T$ and $\|\bm{u}_0\|_{\bm{H}^3(\Omega)}$ for $s \geq 4$.
By virtue of Banach-Alaoglu theorem, we find that there exists a subsequence $\bm{u}^{n_k}$ converges weak-star to some $\bm{u}$ in $L^\infty\left([0, T]; \bm{H}^s(\Omega)\right)$.

We then show that $\bm{u}^n$ converges strongly to $\bm{u}$ in $C([0, T]; \bm{H}^1(\Omega^{\prime}))$ for arbitrary $\Omega^{\prime} \subset\subset \Omega$. Indeed, for fixed $t \in [0, T]$, Rellich-Kondrachov compactness theorem implies that there exists a subsequence of $\bm{u}^{n_k}(t)$ converges strongly to $\bm{u}(t)$ in $\bm{H}^1(\Omega^{\prime})$. Since $[0, T]$ is separable, we can find a dense countable subset $\lbrace{t_l}\rbrace_{l \in \mathbb{N}}$ of $[0, T]$ such that there exists a subsequence of $\bm{u}^{n_k}$ (still denote by $\bm{u}^{n_k}$) satisfies
\begin{align}\label{ub-bk}
\bm{u}^{n_k}(t_l) \rightarrow \bm{u}(t_l) \ \ \text{ strongly in } \bm{H}^1(\Omega^{\prime})
\end{align}
for each $l \in \mathbb{N}$. On the other hand,  multiplying the equation \eqref{second-approximate-1} for $\bm{u}^n$ by $\bm{u}^n$ and integrating over $\Omega$, we have
\begin{align}
\frac{1}{2}\frac{\dif} {\dif t}\|\bm{u}^n\|_{\bm{H}^1(\Omega)}^2 &= -\nu\|\nabla\bm{u}^n\|_{\bm{L}^2(\Omega)}^2 - \nu(1 - ||\vartheta_n||_{L^\infty})\|\bm{u}^n\|_{L^2(\Omega)}^2. \label{ub-ut}
\end{align}
 Observing the right hand side of the above identity is uniformly integrable in [0, T], we conclude that $\{\bm{u}^{n_k}\}$ is equi-continuous in $C([0, T]; \bm{H}^1(\Omega^{\prime}))$, then thanks to Arzel$\grave{a}$-Ascoli theorem,  we obtain that $\bm{u}^{n_k}$ converges to $\bm{u}$ strongly in $C([0, T]; \bm{H}^1(\Omega^{\prime}))$.

 Recalling that $\Omega^{\prime}$ is arbitrary and $\bm{u}^{n_k}$ also converges weak-star to $\bm{u}$ in $L^\infty([0, T]; \bm{H}^s(\Omega))$, we can conclude that the approximate equations \eqref{second-approximate-1}--\eqref{second-approximate-5} for $\bm{u}^{n_k}$ converges to second-grade equations \eqref{second-grade-rewriten-1}--\eqref{second-grade-rewriten-5} for $\bm{u}$ and the identify $\eqref{energy-formula}$ holds.  Since $\Omega^{\prime}$ is arbitrary, it's readily checked by contradiction that $\bm{u} \in C([0, T]; \bm{H}^1(\Omega))$.  Furthermore, we observe that $\bm{u}^{n}$ is uniformly bounded by \eqref{ub-u}, it follows that \eqref{theorem-bound} holds. The proof of Theorem \ref{theorem-1} is completed.

\end{proof}

\section*{Comments and conclusions}
The main motivaton of our recent work is to consider the vanishing viscosity limit problem of Navier-Stokes equations. As we know, in a domain with Dirichlet boundary conditions, this problem remains wide open, something naturally associated with the physical phenomena of turbulence and of boundary layers. Marsden, Ebin, and Fischer \cite{1972Diffeomorphism} even suspected that although in a region with boundary, solutions of the Navier-Stokes equations would not in general converge to the solutions of Euler equations, a certain averaged quantity of the flow may converge. Recently, their conjecture is partly verified. In fact, for 2D bounded domain with Dirichlet boundary conditions, the authors \cite{lopes2015approximation} have showed that second-grade equations converges to Euler equations as the viscosity and $\alpha$ tends to zero. With our main therem about well-posedness of second-grade fuid equations established, one would like to extend the singular limit problem to the case in exterior domain.

At last, let us consider some questions naturally associated with the work we have presented. First, in the main theorem, we assume the initial data $\bm{u}_0$ belongs to $\bm{H}^3(\Omega) \cap V$. In future, we would like to weaken the regularity of $\bm{u}_0$ to get a weaker solution of Euler-$\alpha$ equations in exterior domain. Second, we plan to consider second-grade fluid equations in three dimension, focusing on well-posedness and singular limit problems. Comparing with two dimensional case, the global well-posedness problem in three dimension is more challenging, since the vorticity equation includes a deformation term.

\section*{Acknowledgement}
The work of Aibin Zang was supported in part  by the National Natural Science Foundation of China (Grant no. 11771382, 12061080).
\vskip 0.2cm{}

{}

\normalem
\bibliographystyle{siam}
\bibliography{mybib}

\end{document}